\documentclass[12pt, reqno]{amsart}
\usepackage{amsmath, amsthm, amscd, amsfonts, amssymb, graphicx, color}
\usepackage[bookmarksnumbered, colorlinks, plainpages]{hyperref}
\hypersetup{colorlinks=true,linkcolor=red, anchorcolor=green, citecolor=cyan, urlcolor=red, filecolor=magenta, pdftoolbar=true}

\newtheorem{theorem}{Theorem}[section]
\newtheorem{lemma}[theorem]{Lemma}

\newtheorem{corollary}[theorem]{Corollary}
\theoremstyle{definition}

\theoremstyle{remark}
\newtheorem{remark}[theorem]{Remark}
\numberwithin{equation}{section}

\begin{document}

\setcounter{page}{1}

\title[Derivations and 2-local derivations]{Derivations and 2-local derivations on matrix algebras and  algebras of locally measurable operators}

\author[W. Huang, J. Li, \MakeLowercase{and}  W. Qian]{Wenbo Huang,$^1$  Jiankui  Li,$^1$$^{*}$ \MakeLowercase{and} Wenhua Qian$^2$}

\address{$^{1}$Department of Mathematics, East China University of Science and Technology, Shanghai 200237, China.}
\email{\textcolor[rgb]{0.00,0.00,0.84}{huangwenbo2015@126.com;
jiankuili@yahoo.com}}

\address{$^{2}$College of Mathematical Sciences, Chongqing Normal University, Shapingba District, Chongqing 401331, China.}
\email{\textcolor[rgb]{0.00,0.00,0.84}{whqian86@163.com}}



\date{Received: xxxxxx; Revised: yyyyyy; Accepted: zzzzzz.
\newline \indent $^{*}$Corresponding author}

\begin{abstract}
Let $\mathcal{A}$ be a unital algebra over $\mathbb{C}$ and $\mathcal{M}$ be a unital  $\mathcal{A}$-bimodule.   We show that every derivation $D: M_{n}(\mathcal{A})\rightarrow M_{n}(\mathcal{M}),$ $n\geq2,$  can be represented as a sum $D=D_{m}+\overline{\delta},$ where $D_{m}$ is an inner derivation  and $\overline{\delta}$ is a derivation  induced by a derivation $\delta$ from $\mathcal{A}$ into $\mathcal{M}.$ If $\mathcal{A}$   commutes with $\mathcal{M}$ we prove that every  2-local inner  derivation $\Delta: M_{n}(\mathcal{A}) \to M_{n}(\mathcal{M})$, $n \ge 2$, is an inner derivation. In addition, If $\mathcal{A}$  is commutative and  commutes with $\mathcal{M},$ then every 2-local derivation $\Delta: M_{n}(\mathcal{A}) \to M_{n}(\mathcal{M})$, $n \ge 2$, is a derivation. Let $\mathcal{R}$ be  a finite  von Neumann algebra of type $\textrm{I}$ with  center $\mathcal{Z}$ and $LS(\mathcal{R})$ be the algebra of locally measurable operators affiliated with $\mathcal{R}.$ We also prove that if  the lattice $\mathcal{Z}_{\mathcal{P}}$ of all  projections in $\mathcal{Z}$ is an atomic, then every derivation $D:\mathcal{R}\rightarrow LS(\mathcal{R})$ is an inner derivation.
\end{abstract} \maketitle

\section{Introduction}
Let $\mathcal{A}$ be an algebra over $\mathbb{C}$ the field of complex numbers and $\mathcal{M}$ be an $\mathcal{A}$-bimodule. A linear map $\delta$ from $\mathcal{A}$ into $\mathcal{M}$ is called a \emph{Jordan derivation} if $\delta(a^{2})=\delta(a)a+a\delta(a)$ for each $a$ in $\mathcal{A}$. A linear map $\delta$ from $\mathcal{A}$ into $\mathcal{M}$ is called a\emph{ derivation} if $\delta(ab)=\delta(a)b+a\delta(b)$ for each $a,b$ in $\mathcal{A}$. Let $m$ be an element in $\mathcal{M},$ the map $\delta_{m}:\mathcal{A}\rightarrow \mathcal{M}, ~a\rightarrow \delta_{m}(a):=ma-am,$ is a derivation. A derivation $\delta:\mathcal{A}\rightarrow \mathcal{M}$
 is said to be an \emph{inner derivation} when it can be written in the form $\delta=\delta_{m}$ for some $m$ in $\mathcal{M}.$  A fundamental result, due to Sakai \cite {SA}, states that every derivation on a von Neumann algebra is an inner derivation.

An algebra $\mathcal{A}$ is called  \emph {regular} (in the sense of von Neumann) if for each $a$ in $\mathcal{A}$ there exists $b$ in $\mathcal{A}$ such that $a=aba.$
Let $\mathcal{R}$ be a von Neumann algebra. We denote $S(\mathcal{R})$ and $LS(\mathcal{R})$ respectively the algebras of all measurable and locally measurable operators affiliated with $\mathcal{R}.$ For a faithful normal semi-finite trace $\tau$ on $\mathcal{R},$ we denote the algebra of all $\tau$-measurable operators from $S(\mathcal{R})$ by $S(\mathcal{R},\tau)$
(cf. \cite{AAK1, AK2, N}). If $\mathcal{R}$ is an abelian von Neumann algebra then it is $*$-isomorphic to the algebra $L^{\infty}(\Omega)=L^{\infty}(\Omega,\Sigma,\mu)$
of all (classes of equivalence of) essentially bounded measurable complex functions on a measurable  space $(\Omega,\Sigma,\mu)$  and therefore, $LS(\mathcal{R})=S(\mathcal{R})\cong L^{0}(\Omega),$ where $L^{0}(\Omega)=L^{0}(\Omega,\Sigma,\mu)$ is a unital commutative regular  algebra of all measurable complex functions on $(\Omega,\Sigma,\mu).$  In this case
inner derivations on  the algebra $S(\mathcal{R})$ are identically zero, i.e. trivial.

In \cite{BCS1} Ber, Chilin and Sukochev obtain necessary and sufficient conditions for existence of non trivial derivations on commutative regular algebras. In particular they prove that the algebra $L^{0}(0,1)$ of all measurable complex functions on the interval $(0,1)$  admits non trivial derivations. Let $\mathcal{R}$ be a properly infinite von Neumann algebra. In \cite{AK2}, Ayupov and  Kudaybergenov show that every derivation on the algebra $LS(\mathcal{R})$ is an inner derivation.

In 1997, $\breve{S}$emrl \cite{SE}  introduced 2-local derivations and 2-local automorphisms. A map $\Delta : \mathcal{A} \to \mathcal{M}$ (not necessarily linear) is called a \emph {2-local derivation} if, for every $x, y \in \mathcal{A}$, there exists a derivation $D_{x, y}: \mathcal{A} \to \mathcal{M}$ such that $D_{x, y}(x)=\Delta(x)$ and $D_{x,y}(y) = \Delta(y)$. In particular, if, for every $x, y \in \mathcal{A},$ $D_{x, y}$ is an inner derivation,  then we call $\Delta$ is a \emph{2-local inner derivation.} In \cite{NP1} Niazi and Peralta
introduce the notion of weak-2-local derivation (respectively,  $^{*}$-derivation) and prove that every weak-2-local $^{*}$-derivation on $M_{n}$ is a derivation.
2-local derivations and weak-2-local derivations have been investigated by many authors on different algebras and many results have been obtained in \cite{AK1, AK2,AKA1,AKA2,AA, HLAH, KK, NP1,NP2,SE, ZH}.

Let $\mathcal{H}$ be a infinite-dimensional separable Hilbert space. In \cite{SE} $\breve{S}$emrl shows that every 2-local derivation  on $\mathcal{B}(\mathcal{H})$ is a derivation. In \cite{KK} Kim and Kim give a short proof of  that every 2-local derivation on a finite-dimensional complex matrix algebra is a derivation. In \cite{AK1} Ayupov and Kudaybergenov  extend this result to an arbitrary von Neumann algebra. In \cite{AKA1} Ayupov, Kudaybergenov and Alauadinov prove that if $\mathcal{R}$ is a finite von Neumann algebra  of type $\textrm{I}$ without abelian direct summands, then each 2-local derivation on the algebra $LS(\mathcal{R})=S(\mathcal{R})$  is a derivation.  In the same paper, the authors also show that if $\mathcal{R}$ is an abelian von Neumann algebra such that the lattice of all projections in $\mathcal{R}$ is not atomic, then there exists  a 2-local derivation on the algebra $S(\mathcal{R})$ which is not a derivation. In \cite{ZH} Zhang and Li construct an example of a 2-local derivation on the algebra of all triangular complex $2\times 2$ matrices which is not a derivation.

In \cite{AKA1} Ayupov, Kudaybergenov and Alauadinov show that if $\mathcal{A}$ is a  unital  commutative regular algebra, then every 2-local derivation on the algebra $M_{n}(\mathcal{A}),$  $n\geq 2,$  is a derivation. In \cite{AA} Ayupov and Arzikulov show that if $\mathcal{A}$ is a  unital commutative ring, then every  2-local inner derivation on $M_{n}(\mathcal{A}), ~n\geq 2,$  is an inner derivation. Let $\mathcal{A}$ be a unital Banach algebra and $\mathcal{M}$ be a unital $\mathcal{A}$-bimodule. In \cite{HLAH} He, Li, An and Huang   prove that if every Jordan derivation from $\mathcal{A}$ into $\mathcal{M}$ is an inner derivation then every  2-local derivation from $M_{n}(\mathcal{A})$ $(n\geq 3)$ into $M_{n}(\mathcal{M})$ is a derivation.

Throughout this paper, $\mathcal{A}$ is an algebra with unit $1$ over $\mathbb{C}$ and $\mathcal{M}$ is a unital  $\mathcal{A}$-bimodule. We say that \emph {$\mathcal{A}$ commutes with $\mathcal{M}$} if $am=ma$ for every  $a\in\mathcal{A}$ and  $m\in\mathcal{M}$. From now on, $M_{n}(\mathcal{A})$, for $n \geq 2,$ will denote
the algebra of all $n\times n$ matrices over $\mathcal{A}$ with the usual operations. By the way, we denote any element in $M_{n}(\mathcal{A})$ by $(a_{rs})_{n\times n},$ where $r,s\in \{ 1,2,\ldots,n \};$  $E_{ij},$ $i,j\in \{ 1,2,\ldots,n \},$ the matrix units in $M_{n}(\mathbb{C})$; and
$x\otimes E_{ij},$ the matrix whose $(i,j)$-th entry is $x$ and zero elsewhere. We use $A_{ij}$ for the $(i,j)$-th entry of $A\in M_{n}(\mathcal{A})$ and denote $diag(x_{1},\ldots, x_{n})$ or $diag(x_{i})$ the diagonal matrix with entries $x_{i}\in\mathcal{A}$, $i\in {\{1,2,\ldots,n\}},$  in the diagonal positions.

Let  $\delta:\mathcal{A}\rightarrow\mathcal{M}$ be a derivation. Setting
\begin{align}
\overline{\delta}((a_{ij})_{n\times n})=(\delta(a_{ij}))_{n\times n},\label{701}
\end{align}
we obtain a well-defined linear operator from $M_{n}(\mathcal{A})$ into $M_{n}(\mathcal{M}),$ where  $M_{n}(\mathcal{M})$ has a natural structure of $M_{n}(\mathcal{A})$-bimodule.
Moreover $\overline{\delta}$ is a derivation from $M_{n}(\mathcal{A})$ into $M_{n}(\mathcal{M}).$ If $\mathcal{A}$ is a commutative algebra, then  the restriction of $\overline{\delta}$ onto the center of the algebra $M_{n}(\mathcal{A})$ coincides with the given $\delta.$

In this paper we give characterizations of derivations, 2-local  inner derivations and 2-local derivations from $M_{n}(\mathcal{A})$ into $M_{n}(\mathcal{M})$. In Section 2, we  show that a derivation $D: M_{n}(\mathcal{A})\rightarrow M_{n}(\mathcal{M}),$ $n\geq2$, can be decomposed as a sum of an inner derivation and a derivation induced by a derivation from $\mathcal{A}$ to $\mathcal{M}$ as \eqref{701}, as follows:
$$D=D_{B}+\overline{\delta}.$$
In addition, the  representation of the above form is unique if and only if  $\mathcal{A}$ commutes with $\mathcal{M}$.
 Let $\mathcal{R}$ be  a finite  von Neumann algebra of type $\textrm{I}$ with  center $\mathcal{Z}$ and $LS(\mathcal{R})$ be the algebra of locally measurable operators affiliated with $\mathcal{R}.$ we prove that if  the lattice $\mathcal{Z}_{\mathcal{P}}$ of all  projections in $\mathcal{Z}$ is an atomic, then every derivation $D:\mathcal{R}\rightarrow LS(\mathcal{R})$ is an inner derivation.

In Section 3, we consider 2-local inner derivations and 2-local derivations from $M_{n}(\mathcal{A})$ into $M_{n}(\mathcal{M})$. For the case that $\mathcal{A}$ commutes with $\mathcal{M}$, we obtain that every inner 2-local derivation from $M_{n}(\mathcal{A})$ into $M_{n}(\mathcal{M})$ is an inner derivation.
In addition,  if $\mathcal{A}$ is commutative, we prove that every 2-local derivation $\Delta: M_{n}(\mathcal{A})\rightarrow M_{n}(\mathcal{M})$, $n\geq2$, is a derivation. Let $\mathcal{R}$ be an arbitrary von Neumann algebra without abelian direct summands. We also show  every 2-local derivation $\Delta: \mathcal{R}\rightarrow LS(\mathcal{R})$ is a derivation.

\section{Derivations}

Let $\mathcal{A}$ be an algebra with unit $1$ over $\mathbb{C}$ and $\mathcal{M}$ be a unital $\mathcal{A}$-bimodule. Let  $D: M_{n}(\mathcal{A})\rightarrow M_{n}(\mathcal{M}),$ $n\geq2,$ be a derivation.
Firstly, we define a map $D^{ij}_{rs}: \mathcal{A}\rightarrow\mathcal{M}$ by
$$D^{ij}_{rs}(a)=[D(a\otimes E_{rs})]_{ij}, ~a\in \mathcal{A},~i,j,r,s\in{\{1,2,\ldots,n\}}.$$
For any $a,b\in \mathcal{A}$  and some fixed $m\in{\{1,2,\ldots,n\}}, $ we have
\begin{align*}
D^{ij}_{rs}(ab)
&=[D(ab\otimes E_{rs})]_{ij}\notag\\
&=[D((a\otimes E_{rm})(b\otimes E_{ms}))]_{ij}\notag\\
&=[D(a\otimes E_{rm})(b\otimes E_{ms})]_{ij}+[(a\otimes E_{rm})D(b\otimes E_{ms})]_{ij}\notag\\
&=\delta_{js}[D(a\otimes E_{rm})]_{im}b+\delta_{ir}a[D(b\otimes E_{ms})]_{mj},
\end{align*}
where $\delta$ is the Kronecker's delta.  It follows that
\begin{align}
D^{ij}_{rs}(ab)=\delta_{js}[D(a\otimes E_{rm})]_{im}b+\delta_{ir}a[D(b\otimes E_{ms})]_{mj}. \label{705}
\end{align}
 For any  $m\in{\{1,2,\ldots,n\}},$  we deduce from the equality  \eqref{705} that,
$$D^{mm}_{mm}(ab)=D^{mm}_{mm}(a)b+aD^{mm}_{mm}(b),$$
thus  $D^{mm}_{mm}:\mathcal{A}\rightarrow \mathcal{M}$ is a derivation.   We abbreviate the derivation $D^{mm}_{mm}$  by $D^{m}.$ Particularly, we denote
the derivation $D^{11}_{11}$ by $D^{1}.$

\begin{theorem} \label{thm 2.1}
Every derivation $D: M_{n}(\mathcal{A})\rightarrow M_{n}(\mathcal{M}),$ $n\geq2,$  can be represented as a sum
\begin{align}
D=D_{B}+\overline{\delta}, \label{703}
\end{align}
where $D_{B}$ is an inner derivation implemented by an element $B\in M_{n}(\mathcal{M})$ and $\overline{\delta}$ is a derivation of the form \eqref{701} induced by a derivation $\delta$ from $\mathcal{A}$ into $\mathcal{M}$. Furthermore, if this representation is unique for every derivation $D$, then $\mathcal{A}$ commutes with $\mathcal{M}$ (i.e. $am=ma$ for every $a\in\mathcal{A},~m\in\mathcal{M}$); and if  $\mathcal{A}$ commutes with $\mathcal{M}$ then this representation is always unique.
\end{theorem}

Before the proof of Theorem \ref{thm 2.1}, we first present the following lemma.
\begin{lemma} \label{lem 2.2}
For every  $~i,j,r,s,m\in {\{1,2,\ldots,n\}}$ and every $a\in\mathcal{A}$ the following equalities hold:
\begin{enumerate}
\item[{\rm(i)}] $D^{ij}_{rs}=0,$  $i\neq r$  and  $j\neq s,$
\item[{\rm(ii)}] $D^{ij}_{rj}(a)=D^{im}_{rm}(a)=D^{im}_{rm}(1)a,$  if  $i\neq r,$
\item[{\rm(iii)}] $D^{ij}_{is}(a)=D^{mi}_{ms}(a)=aD^{mj}_{ms}(1),$  if  $j\neq s,$
\item[{\rm(iv)}] $D^{im}_{jm}(1)=-D^{mj}_{mi}(1),$
\item[{\rm(v)}]  $D^{ij}_{ij}(a)=D^{im}_{im}(1)a-aD^{jm}_{jm}(1)+D^{m}(a).$
\end{enumerate}
\end{lemma}

\begin{proof}
It obviously follows from \eqref{705} that, statements $\rm(i)$, $\rm(ii)$ and $\rm(iii)$  hold.   We only need to prove $\rm(iv)$ and $\rm(v).$

$\rm(iv)$: In the case $i=j,$ we have
\begin{align*}
0&=[D(1\otimes E_{ii})]_{ii}=[D((1\otimes E_{im})(1\otimes E_{mi}))]_{ii}\notag\\
&=[D((1\otimes E_{im}))(1\otimes E_{mi})]_{ii}+[(1\otimes E_{im})D((1\otimes E_{mi}))]_{ii}\notag\\
&=D^{im}_{im}(1)+D^{mi}_{mi}(1),\notag
\end{align*}
i.e.
\begin{align}
D^{im}_{im}(1)=-D^{mi}_{mi}(1).\label{707}
\end{align}
For the case $i\neq j,$ we have
\begin{align*}
0&= D(0)=[D((1\otimes E_{ii})(1\otimes E_{jj}))]_{ij}\notag\\
&=[D((1\otimes E_{ii}))(1\otimes E_{jj})]_{ij}+[(1\otimes E_{ii})D((1\otimes E_{jj}))]_{ij}\notag\\
&=[D(1\otimes E_{ii})]_{ij}+[D(1\otimes E_{jj})]_{ij}\notag\\
&=D^{ij}_{ii}(1)+D^{ij}_{jj}(1),\notag
\end{align*}
i.e. $$D^{ij}_{jj}(1)=-D^{ij}_{ii}(1).$$
By $\rm(ii)$ , $\rm(iii)$ and equality \eqref{707},  it follows that
$$D^{im}_{jm}(1)=-D^{mj}_{mi}(1).$$
$\rm(v)$: By equality \eqref{705}, we have
\begin{align}
D^{ij}_{ij}(a)=D^{im}_{im}(1)a+D^{mj}_{mj}(a),\label{711}
\end{align}
and
\begin{align}
D^{ij}_{ij}(a)=D^{im}_{im}(a)+aD^{mj}_{mj}(1).\label{713}
\end{align}
Taking $j=m$ in equality \eqref{711}, we obtain that
\begin{align}
D^{im}_{im}(a)=D^{im}_{im}(1)a+D^{m}(a).\label{715}
\end{align}
By equalities \eqref{707}, \eqref{713} and \eqref{715}, it follows that
$$D^{ij}_{ij}(a)=D^{im}_{im}(1)a-aD^{jm}_{jm}(1)+D^{m}(a).$$
The proof is complete.
\end{proof}
Now we are in position to prove Theorem \ref{thm 2.1}.
\begin{proof}[Proof of Theorem \ref{thm 2.1}]
Let $(a_{rs})_{n\times n}$ be an arbitrary element in $M_{n}(\mathcal{A})$ and $D$ be a derivation from  $M_{n}(\mathcal{A})$ into $M_{n}(\mathcal{M})$. For any  $i,j\in {\{1,2,\ldots,n\}},$ it follows from Lemma \ref{lem 2.2} that
\begin{align*}
[D((a_{rs})_{n\times n})]_{ij}&=\sum^{n}_{r,s=1}D^{ij}_{rs}(a_{rs})\notag\\
&=\sum^{n}_{r=1}D^{ij}_{rj}(a_{rj})+\sum^{n}_{s=1}D^{ij}_{is}(a_{is})-D^{ij}_{ij}(a_{ij})\notag\\
&=\sum_{r\neq i}D^{ij}_{rj}(a_{rj})+\sum_{s\neq j}D^{ij}_{is}(a_{is})+D^{ij}_{ij}(a_{ij})\notag\\
&=\sum_{r\neq i}D^{i1}_{r1}(1)a_{rj}+\sum_{s\neq j}a_{is}D^{1j}_{1s}(1)+D^{i1}_{i1}(1)a_{ij}-a_{ij}D^{j1}_{j1}(1)+D^{1}(a_{ij})\notag\\
&=\sum^{n}_{r=1}D^{i1}_{r1}(1)a_{rj}-\sum^{n}_{s=1}a_{is}D^{s1}_{j1}(1)+D^{1}(a_{ij})\notag\\
&=\sum^{n}_{k=1}(D^{i1}_{k1}(1)a_{kj}-a_{ik}D^{k1}_{j1}(1))+D^{1}(a_{ij})\notag\\
&=[(D^{r1}_{s1}(1))_{n\times n}(a_{rs})_{n\times n}-(a_{rs})_{n\times n}(D^{r1}_{s1}(1))_{n\times n}]_{ij}+[\overline{D^{1}}((a_{rs})_{n\times n})]_{ij},\notag
\end{align*}
i.e.
\begin{align}
[D((a_{rs})_{n\times n})]_{ij}=[(D^{r1}_{s1}(1))_{n\times n}(a_{rs})_{n\times n}-(a_{rs})_{n\times n}(D^{r1}_{s1}(1))_{n\times n}]_{ij}+[\overline{D^{1}}((a_{rs})_{n\times n})]_{ij},\label{719}
\end{align}
where $(D^{r1}_{s1}(1))_{n\times n}\in M_{n}(\mathcal{M})$ and $[(D^{r1}_{s1}(1))_{n\times n}]_{ij}=D^{i1}_{j1}(1).$
By equality \eqref{719}, we have
$$D((a_{rs})_{n\times n})=[(D^{r1}_{s1}(1))_{n\times n}(a_{rs})_{n\times n}-(a_{rs})_{n\times n}(D^{r1}_{s1}(1))_{n\times n}]+[\overline{D^{1}}((a_{rs})_{n\times n})].$$
We denote $B=(D^{r1}_{s1}(1))_{n\times n}$ and  $\delta=D^{1}.$ Therefore every derivation $D: M_{n}(\mathcal{A})\rightarrow M_{n}(\mathcal{M}),$ $n\geq2,$  can be represented as a sum
$$D=D_{B}+\overline{\delta}.$$

Suppose that $D_{M}$ is an inner derivation from $M_{n}(\mathcal{A})$ into $M_{n}(\mathcal{M})$  implemented by an element $M\in M_{n}(\mathcal{M}),$ and $\overline{\zeta}$ is a derivation of the form \eqref{701} induced by a derivation $\zeta$ from $\mathcal{A}$ into $\mathcal{M},$ such that $D_{M}=\overline{\zeta}.$ The first step is to establish the following.

$\textbf{Claim~1}.$  If $\mathcal{A}$ commutes with $\mathcal{M}$, then $D_{M}=\overline{\zeta}=0.$

{Proof of Claim 1~} If $i\neq j,~i,j\in {\{1,2,\ldots,n\}},$ we have
$$0=\overline{\zeta}(E_{ij})=D_{M}(E_{ij})=ME_{ij}-E_{ij}M.$$
It follows that $M_{ji}=0.$  Thus $M$ has a diagonal form, i.e. $M=diag(M_{kk}).$
Suppose that $\overline{\zeta}\neq 0,$  then there exists an element $a\in \mathcal{A}$ such that $\zeta(a)\neq0.$
Take $A=diag(a),$ then $ \overline{\zeta}(A)\neq0.$ On the other hand,
$$\overline{\zeta}(A)=D_{M}(A)=diag(M_{kk})diag(a)-diag(a)diag(M_{kk})=0.$$
This is a contradiction. Thus $~\overline{\zeta}=0.$

$\textbf{Claim~2}.$ If $\mathcal{A}$ does not commute with $\mathcal{M}$, then there exist $D_{M}$ and $\overline{\zeta},$ such that ~$D_{M}=\overline{\zeta}\neq0.$

{Proof of Claim 2~}  By assumption, we can take  $a\in\mathcal{A}$ and $ m\in\mathcal{M}$  such that $ma\neq am.$
We define a derivation $\zeta:\mathcal{A}\rightarrow\mathcal{M}$  by
$\zeta(x)=mx-xm$  for every $x$ in $\mathcal{A}.$
We denote $M=diag(m)\in M_{n}(\mathcal{M}),$  then $D_{M}$  is an inner derivation from $ M_{n}(\mathcal{A})$  into $ M_{n}(\mathcal{M}).$
Obviously, $D_{M}=\overline{\zeta}~$  and $~\overline{\zeta}(diag(a))\neq0.$
Thus $D_{M}=\overline{\zeta}\neq0.$

In the following,  we show that the representation of the above form is unique if and only if  $\mathcal{A}$ commutes with $\mathcal{M}.$

$\mathbf{Case~1}.$ If $\mathcal{A}$ commutes with $\mathcal{M},$ we suppose that there exists a derivation $D:M_{n}(\mathcal{A})\rightarrow M_{n}(\mathcal{M}),$ $n\geq2,$ which
can be represented as $D=D_{B_{1}}+\overline{\delta_{1}}=D_{B_{2}}+\overline{\delta_{2}}.$ This means that $D_{B_{1}}-D_{B_{2}}=\overline{\delta_{2}}-\overline{\delta_{1}}.$
Since  $D_{B_{1}}-D_{B_{2}}=D_{B_{1}-B_{2}}$ and $\overline{\delta_{2}}-\overline{\delta_{1}}=\overline{\delta_{2}-\delta_{1}},$
we have $D_{B_{1}-B_{2}}=\overline{\delta_{2}-\delta_{1}}.$
It follows from Claim 1 that, $D_{B_{1}-B_{2}}=\overline{\delta_{2}-\delta_{1}}=0.$ i.e. $D_{B_{1}}=D_{B_{2}}$
and $\overline{\delta_{1}}=\overline{\delta_{2}}.$

$\mathbf{Case~2}.$ If $\mathcal{A}$ does not commute with $\mathcal{M},$ by Claim 2, there exist derivations $D_{M}$ and $\overline{\zeta}$ from $M_{n}(\mathcal{A})$ into $M_{n}(\mathcal{M}),$
$n\geq2,$ such that $D_{M}=\overline{\zeta}\neq 0.$ Let $D:M_{n}(\mathcal{A})\rightarrow M_{n}(\mathcal{M}),$ $n\geq2,$ be an arbitrary derivation. By hypothesis, $D$ can be represented as $D=D_{B}+\overline{\delta}.$ We have $D=D_{B}+\overline{\delta}=D_{B}+D_{M}-\overline{\zeta}+\overline{\delta}=D_{B+M}+\overline{\delta-\zeta}.$
This means that, the derivation $D$ can be represented as $D=D_{B}+\overline{\delta},$ and as $D=D_{B+M}+\overline{\delta-\zeta}$ too. Therefore, the representation
of  \eqref{703} is not unique for every derivation $D.$

It follows from Cases 1 and 2 that, the  representation of \eqref{703}
is unique if and only if  $\mathcal{A}$ commutes with $\mathcal{M}$.
The proof is complete.
\end{proof}
As applications of Theorem \ref{thm 2.1}, we obtain the following corollaries.
\begin{corollary}\label{cor 2.3}
The following statements are equivalent.
\item[{\rm(i)}] Every derivation $\delta:\mathcal{A}\rightarrow\mathcal{M}$ is an inner derivation.
\item[{\rm(ii)}] Every derivation $D:M_{n}(\mathcal{A})\rightarrow M_{n}(\mathcal{M}),$ $n\geq2,$ is an inner derivation.
\end{corollary}
\begin{proof}
If $\delta:\mathcal{A}\rightarrow\mathcal{M}$ is an inner derivation, by the equality \eqref{701}, obviously, $\overline {\delta}:M_{n}(\mathcal{A})\rightarrow M_{n}(\mathcal{M}),$ $n\geq2,$ is an inner derivation.

${\rm(i)}$ implies ${\rm(ii)}$: Let $D:M_{n}(\mathcal{A})\rightarrow M_{n}(\mathcal{M}),$ $n\geq2,$ be an arbitrary derivation. By Theorem \ref{thm 2.1}, $D$ can be represented as a sum
$D=D_{M}+\overline {\delta},$ where $D_{M}$ is an inner derivation. By hypothesis, $\delta$ is an inner derivation from $\mathcal{A}$ into $\mathcal{M},$ therefore $\overline {\delta}$
is an inner derivation. We known that the sum of two inner derivations is an inner derivation, this means that $D:M_{n}(\mathcal{A})\rightarrow M_{n}(\mathcal{M}),$ $n\geq2,$ is an inner derivation.

${\rm(ii)}$ implies ${\rm(i)}$: Suppose that $\delta$ is a derivation from $\mathcal{A}$ into $\mathcal{M},$  then $\overline {\delta}:M_{n}(\mathcal{A})\rightarrow M_{n}(\mathcal{M}),$ $n\geq2,$ is a derivation. By hypothesis, $\overline {\delta}$ is an inner derivation. then the restriction of $\overline {\delta}$ onto $E_{11}M_{n}(\mathcal{A})E_{11},$  the subalgebra of $M_{n}(\mathcal{A}),$  is an inner derivation.  This means that $\delta:\mathcal{A}\rightarrow\mathcal{M}$ is an inner derivation.
\end{proof}
\begin{corollary}
Let $\mathcal{A}$ be a commutative unital algebra over $\mathbb{C}.$ Then every derivation on the matrix algebra $M_{n}(\mathcal{A})~(n\geq2)$ is inner if and only if every derivation on $\mathcal{A}$ is identically zero, i.e. trivial.
\end{corollary}
Let $\mathcal{R}$ be  a von Neumann algebra.  Denote by $S(\mathcal{R})$ and $LS(\mathcal{R})$ respectively the sets of all measurable and locally measurable operators affiliated with $\mathcal{R}.$ Then the  set $LS(\mathcal{R})$ of all locally measurable operators with respect to $\mathcal{R}$ is a unital $*$-algebra when equipped with the algebraic operations of strong addition and multiplication and taking the adjoint of an operator and   $S(\mathcal{R})$ is a solid $*$-subalgebra in $LS(\mathcal{R})$. If $\mathcal{R}$ is a finite von Neumann algebra,  then $S(\mathcal{R})=LS(\mathcal{R})$ (see, for example, \cite{AAK1, AK2, N}).
 Let $\mathcal{A}$  be a commutative  algebra with unit $1$ over  $\mathbb{C}.$  We denote by $\nabla$ the set $\{e\in\mathcal{A}:e^{2}=e\}$
 of all idempotents in $\mathcal{A}.$ For $e,f\in\nabla$ we set $e\leq f$ if $ef=e.$ Equipped with this partial order, lattice operations $e\vee f=e+f-ef,$ $e\wedge f=ef$ and the
 complement $e^{\perp}=1-e,$ the set $\nabla$ forms a Boolean algebra. A non zero element $q$ from the Boolean algebra  $\nabla$ is called an atom if $0\neq e\leq q,$ $e\in\nabla,$ imply that $e=q.$  If given any nonzero $e\in \nabla$ there exists an atom $q$ such that $q\leq e,$ then the Boolean algebra $\nabla$ is said to be \emph{atomic}.

 Let $\mathcal{R}$ be an abelian von Neumann algebra. Theorem 3.4 of \cite{BCS1} implies that every derivation on the algebra $S(\mathcal{R})$ is inner if and only if the lattice $\mathcal{R}_{\mathcal{P}}$ of all  projections in $\mathcal{R}$ is atomic. If $\mathcal{R}$ is a properly infinite von Neumann algebra,  in \cite{AK2} the authors show that every derivation on the algebra $LS(\mathcal{R})$ is inner (see \cite{AK2}, Theorem 4.6). In the case of  $\mathcal{R}$ is a finite von Neumann algebra of  type $\textrm{I}$, Theorem 3.5 of \cite{AK2} shows that a derivation on the algebra $LS(\mathcal{R})$ is an inner derivation if and only if it is identically zero on the center of  $\mathcal{R}.$

As a direct application of Corollary \ref{cor 2.3}, we  obtain the following result.
\begin{corollary}
Let $\mathcal{R}$ be  a finite  von Neumann algebra of type $\textrm{I}$  with  center $\mathcal{Z}.$ Then
every derivation {D} on the algebra $LS(\mathcal{R})$ is inner if and only if the lattice $\mathcal{Z}_{\mathcal{P}}$ of all  projections in $\mathcal{Z}$ is an atomic.
\end{corollary}
\begin{proof}
Let $\mathcal{R}$ be  a finite  von Neumann algebra of type $\textrm{I}$  with  center $\mathcal{Z}.$ There exists a family $\{e_{n}\}_{n\in\mathcal{F}},$ $\mathcal{F}\subseteq \mathbb{N},$
of central projections from $\mathcal{R}$ with  $\bigvee \limits _{n\in\mathcal{F}}e_{n}=1$ such that the algebra $\mathcal{R}$ is $*$-isomorphic with the $C^{*}$-product of von Neumann algebras $e_{n}\mathcal{R}$ of type $\textrm{I}_{n}$ respectively, $n\in\mathcal{F},$ i.e.
 $\mathcal{R} \cong\bigoplus \limits _{n\in\mathcal{F}}e_{n}\mathcal{R}.$ By Proposition 1.1 of \cite{AAK1}, we have that
 $LS(\mathcal{R}) \cong\prod \limits _{n\in\mathcal{F}}LS(e_{n}\mathcal{R}).$

Suppose that $D$ is a derivation on $LS(\mathcal{R})$ and $\delta$  its restriction onto the center $S(\mathcal{Z}).$ Since $\delta$ maps each $e_{n}S(\mathcal{Z})$ into  itself,
$\delta$ generates a derivation $\delta_{n}$ on $e_{n}S(\mathcal{Z})$ for each $n\in\mathcal{F}.$ By Proposition 1.5 of \cite{AAK1},
 $LS(e_{n}\mathcal{R})\cong M_{n}(e_{n}S(\mathcal{Z}))$. Let $\overline{\delta}_{n}$ be the derivation on the matrix algebra
$M_{n}(e_{n}S(\mathcal{Z}))$ defined as in \eqref{701}. Put
\begin{align}
\overline{\delta}(\{x_{n}\}_{n\in\mathcal{F}})=\{\overline{\delta}_{n}(x_{n})\}, ~\{x_{n}\}_{n\in\mathcal{F}}\in LS(\mathcal{R}). \label{1801}
\end{align}
Then the map $\overline{\delta}$ is a derivation on $LS(\mathcal{R}).$  Lemma 2.3 of \cite{AAK1}  implies that each derivation $D$  on $LS(\mathcal{R})$ can be uniquely represented in the form
$D=D_{B}+\overline{\delta},$ where $D_{B}$ is an inner derivation and $\overline{\delta}$ is a derivation given as \eqref{1801}.

 If $D$ is an arbitrary derivation on $LS(\mathcal{R})$ and $\delta$  its restriction onto center $S(\mathcal{Z}),$ by Theorem 3.4 of \cite{BCS1}, the lattice $\mathcal{Z}_{\mathcal{P}}$  is an atomic if and only if $\delta=0.$ We have $\delta=0$ if and only if
$\delta_{n}=0$ for each $n\in\mathcal{F}.$  By Corollary \ref{cor 2.3}, $\delta_{n}=0$ if and only if $\overline{\delta_{n}}=0$ for each $n\in\mathcal{F}.$  By equality \eqref{1801}, $\overline{\delta_{n}}=0$ for each $n\in\mathcal{F}$ if and only if $\overline{\delta}=0.$
Therefore, every derivation on the algebra $LS(\mathcal{R})$ is inner derivation if and only if the lattice $\mathcal{Z}_{\mathcal{P}}$ of all  projections in $\mathcal{Z}$ is an atomic.
The proof is complete.
\end{proof}
Let $\mathcal{R}$ be a properly infinite von Neumann algebra and $\mathcal{M}$ be a $\mathcal{R}$-bimodule of locally measurable operators.
In \cite{BCS2}, the authors show that every derivation $D:\mathcal{R}\rightarrow \mathcal{M}$ is an inner derivation.  In the case of $\mathcal{R}$ is a finite von Neumann algebra of  type $\textrm{I}$, we obtain the following result.
\begin{theorem}
 Let $\mathcal{R}$ be  a finite  von Neumann algebra of type $\textrm{I}$ with  center $\mathcal{Z}.$  If  the lattice $\mathcal{Z}_{\mathcal{P}}$ of all  projections in $\mathcal{Z}$ is an atomic, then every derivation $D:\mathcal{R}\rightarrow LS(\mathcal{R})$ is an inner derivation.
\end{theorem}
\begin{proof}
Choose a central decomposition $\{e_{n}\}_{n\in\mathcal{F}},$ $\mathcal{F}\subseteq \mathbb{N},$ of the unity $1$ such that $e_{n}\mathcal{R}$ is a type  $\textrm{I}_{n}$ von Neumann algebra for each $n\in\mathcal{F}.$  By hypothesis, it is easy to check that  $D(e_{n}\mathcal{R})\subseteq e_{n}LS(\mathcal{R})$  for each $n\in\mathcal{F}.$  Thus we only need
to show that the derivation $D$ restricted to $e_{n}\mathcal{R}$ is an inner derivation for each $n\in\mathcal{F}.$

Let $e_{n}\mathcal{R}$ be a type $\textrm{I}_{n} ~(n\in\mathcal{F})$ von Neumann algebra with center $e_{n}\mathcal{Z}.$ It is well known that $e_{n}\mathcal{R}\cong M_{n}(e_{n}\mathcal{Z}).$ We denote the center of $S(e_{n}\mathcal{R})$ by $\mathcal{Z}(S(e_{n}\mathcal{R}))$. By Proposition 1.2 of \cite{AAK1}, we have $\mathcal{Z}(S(e_{n}\mathcal{R}))=S(e_{n}\mathcal{Z}).$  By Proposition 1.5 of \cite{AAK1}, $LS(e_{n}\mathcal{R})=S(e_{n}\mathcal{R})\cong M_{n}(S(e_{n}\mathcal{Z})).$

By assumption, the lattice $\mathcal{Z}_{\mathcal{P}}$ of all  projections in $\mathcal{Z}$ is an atomic. This implies that  the lattice  $e_{n} \mathcal{Z}_{\mathcal{P}}$ is also an atomic for each $n\in\mathcal{F}.$
Statements ${\rm(ii)}$ of Proposition 2.3 and  ${\rm(vi)}$ of Proposition 2.6 of  \cite{BCS1} imply that every derivation $\delta:e_{n}\mathcal{Z}\rightarrow S(e_{n}\mathcal{Z})$ is trivial. By Corollary \ref{cor 2.3},  we have that  every derivation from $M_{n}(e_{n}\mathcal{Z})$  into $M_{n}(S(e_{n}\mathcal{Z}))$ is inner. The proof is complete.
\end{proof}
\section{2-local derivations}

This section is devoted to  2-local inner derivations and 2-local derivations from $M_{n}(\mathcal{A})$ into $M_{n}(\mathcal{M})$. Throughout this section, we  always assume that $\Delta: M_{n}(\mathcal{A})\rightarrow M_{n}(\mathcal{M})$ is a 2-local derivation.
 Firstly, we give the following lemma.
\begin{lemma}\label{lem 3.1}
For every 2-local derivation $\Delta: M_{n}(\mathcal{A})\rightarrow M_{n}(\mathcal{M})$, $n\geq2$, there exists a derivation  $D: M_{n}(\mathcal{A})\rightarrow M_{n}(\mathcal{M})$
such that $\Delta(E_{ij})=D(E_{ij})$  for all  $i,j\in {\{1,2,\ldots,n\}}.$ In particular, if $\Delta$ is a 2-local inner derivation, then $D$ is an inner derivation.
\end{lemma}
\begin{proof}
Let  $\Delta: M_{n}(\mathcal{A})\rightarrow M_{n}(\mathcal{M})$, $n\geq2$, be a 2-local derivation. By Theorem \ref{thm 2.1}, with the proof  similar to the proof of Theorem 3 in \cite{KK}, it is easy to check that there exists a derivation $D$ such that $\Delta(E_{ij})=D(E_{ij})$  for all  $i,j\in {\{1,2,\ldots,n\}}.$

Let $\Delta$  be an inner 2-local derivation.
We define two matrices $S$, $T$ in $M_{n}(\mathcal{A})$  by
$$S=\sum^{n}_{i=1}i1\otimes E_{ii},~~T=\sum^{n-1}_{i=1}E_{ii+1}.$$
By assumption, there exists an inner derivation $D: M_{n}(\mathcal{A})\rightarrow M_{n}(\mathcal{M})$  such that
$$\Delta(S)=D(S),~~\Delta(T)=D(T).$$
Replacing $\Delta$ by $\Delta-D$  if necessary, we may assume that $\Delta(S)=\Delta(T)=0.$
Fixed $i,j\in {\{1,2,\ldots,n\}},$ by assumption, we can take two elements $X,~Y$  in $M_{n}(\mathcal{M})$
  such that
$$\Delta(E_{ij})=XE_{ij}-E_{ij}X,~~0=\Delta(S)=XS-SX,$$
and
$$\Delta(E_{ij})=YE_{ij}-E_{ij}Y,~~0=\Delta(T)=YT-TY.$$

It follows  $XS=SX$ that,  $X$ is a diagonal matrix. We denote $X$ by $diag(x_{k}).$  $YT=TY$ implies that $Y$ is of the form
$$Y=\left[\begin{array}{cccccc}
                     \ y_{1}&y_{2}&y_{3}&\cdot&\cdot&y_{n}\\
                     \ 0&y_{1}&y_{2}&\cdot&\cdot&y_{n-1}\\
                     \ 0&0&y_{1}&\cdot&\cdot&y_{n-2}\\
                     \ \vdots&\vdots&\vdots&\vdots&\vdots&\vdots\\
                     \ \cdots&\cdots&\cdots&\cdot&y_{1}&y_{2}\\
                     \ 0&0&\cdots&\cdot&0&y_{1}\\
                   \end{array}
                 \right].$$
On the one hand, $$\Delta(E_{ij})=XE_{ij}-E_{ij}X=diag(x_{k})E_{ij}-E_{ij}diag(x_{k})=(x_{i}-x_{j})\otimes E_{ij}.$$
On the other hand, $$[\Delta(E_{ij})]_{ij}=[YE_{ij}-E_{ij}Y]_{ij}=[YE_{ij}-E_{ij}Y]_{ij}=0.$$
Therefore  $\Delta(E_{ij})=0.$
The proof is complete.
\end{proof}

\begin{theorem} \label{thm 3.2}
Suppose that $\mathcal{A}$ commutes with $\mathcal{M}.$ Then every 2-local inner derivation $\Delta: M_{n}(\mathcal{A})\rightarrow M_{n}(\mathcal{M})$, $n\geq2$, is an inner derivation.
\end{theorem}

\begin{proof}
By Lemma \ref{lem 3.1}, we may assume that $\Delta(E_{ij})=0$ for all $i,j\in {\{1,2,\ldots,n\}}.$
For any $A\in M_{n}(\mathcal{A}),$ we take a pair $(j,i)$, $j,i\in {\{1,2,\ldots,n\}},$  by assumption, there exists an inner derivation $D_{B}$, such that
$\Delta(A)=D_{B}(A)$ and $0=\Delta(E_{ij})=D_{B}(E_{ij})$. We have
\begin{align*}
E_{ij}\Delta(A)E_{ij}&=E_{ij}D_{B}(A)E_{ij}\notag\\
&=D_{B}(E_{ij}AE_{ij})-D_{B}(E_{ij})AE_{ij}-E_{ij}AD_{B}(E_{ij})=D_{B}(E_{ij}AE_{ij})\notag\\
&=D_{B}(A_{ji}\otimes E_{ij})=D_{B}(diag(A_{ji},\ldots,A_{ji})E_{ij})\notag\\
&=D_{B}(diag(A_{ji},\ldots,A_{ji}))E_{ij}+diag(A_{ji},\ldots,A_{ji})D_{B}(E_{ij})\notag\\
&=(Bdiag(A_{ji},\ldots,A_{ji})-diag(A_{ji},\ldots,A_{ji})B)E_{ij}\notag\\
&=0,
\end{align*}
i.e. $$E_{ij}\Delta(A)E_{ij}=0.$$
Therefore $$E_{ji}(E_{ij}\Delta(A)E_{ij})E_{ji}=E_{jj}\Delta(A)E_{ii}=0,$$
i.e. $$[\Delta(A)]_{ji}=0,$$ for every $j,i\in {\{1,2,\ldots,n\}}.$
Hence $\Delta(A)=0.$ The proof is complete.
\end{proof}
\begin{corollary}
Suppose that $\mathcal{A}$ is a unital  commutative algebra over $\mathbb{C}.$ Then every 2-local inner derivation $\Delta: M_{n}(\mathcal{A})\rightarrow M_{n}(\mathcal{A})$, $n\geq2$, is an inner derivation.
\end{corollary}
\begin{remark}
The above result is proved in \cite{AA}. By comparison, our proof is more simple.
\end{remark}
Suppose that $\mathcal{A}$ is an algebra over $\mathbb{C}$ and $\mathcal{B}$ is a unital subalgebra in $\mathcal{A}.$  We denote the commutant of  $\mathcal{B}$ by $\mathcal{B}^{\prime}=\{a\in\mathcal{A}:ab=ba, \MakeLowercase{for~every~}b\in\mathcal{B}\}$. Let $\mathcal{C}$ be a submodule in $\mathcal{B}^{\prime}.$ It follows from Theorem \ref{thm 3.2} that
\begin{corollary}
 Every 2-local inner derivation $\Delta: M_{n}(\mathcal{B})\rightarrow M_{n}(\mathcal{C})$, $n\geq2$, is an inner derivation.
\end{corollary}
\begin{theorem} \label{thm 3.6}
 Suppose that $\mathcal{A}$ is a commutative algebra which commutes with $\mathcal{M}$.  Then every 2-local derivation $\Delta: M_{n}(\mathcal{A})\rightarrow M_{n}(\mathcal{M})$, $n\geq2$, is a derivation.
\end{theorem}
\begin{proof}
The proof is similar to the proof of Theorem 4.3 in \cite{AKA1}. We leave it to the reader.
\end{proof}
\begin{corollary}
Suppose that $\mathcal{A}$ is a unital  commutative algebra over $\mathbb{C}.$ Then every 2-local  derivation $\Delta: M_{n}(\mathcal{A})\rightarrow M_{n}(\mathcal{A})$, $n\geq2$, is a derivation.
\end{corollary}
If $\mathcal{A}$ is a non commutative algebra, by Theorem \ref{thm 2.1} every derivation from $M_{n}(\mathcal{A})$ into $M_{n}(\mathcal{M})(n\geq2)$  can be represented as a sum
$D=D_{B}+\overline{\delta}.$  In \cite{AKA2}, the authors apply this representation of derivation to prove the following result.
\begin{theorem}[\cite{AKA2}, Theorem 2.1]\label{thm 3.8}
Let $\mathcal{A}$ be a unital Banach algebra and $\mathcal{M}$ be a unital $\mathcal{A}$-bimodule. If every Jordan derivation from $\mathcal{A}$ into $\mathcal{M}$ is a
derivation, then every 2-local derivation $\Delta: M_{n}(\mathcal{A})\rightarrow M_{n}(\mathcal{A})$, $n\geq3$, is a derivation.
\end{theorem}
\begin{theorem}\label{thm 3.9}
Let $\mathcal{A}$ be a unital Banach algebra and $\mathcal{M}$ be a unital $\mathcal{A}$-bimodule. If $n\geq6$ is a positive integer  but not a prime number, then every
 2-local derivation $\Delta: M_{n}(\mathcal{A})\rightarrow M_{n}(\mathcal{M})$ is a derivation.
\end{theorem}
\begin{proof}
Suppose that $n=rk$, where $r\geq3$ and $k\geq2.$ Then $M_{n}(\mathcal{A})\cong M_{r}(M_{k}(\mathcal{A}))$ and  $M_{n}(\mathcal{M})\cong M_{r}(M_{k}(\mathcal{M})).$
In \cite{A}, the author proves that every Jordan derivation from $M_{k}(\mathcal{A})$ into $M_{k}(\mathcal{M})(k\geq2)$  is a derivation (\cite{A}, Theorem 3.1). By Theorem \ref{thm 3.8}, the proof is complete.
\end{proof}
 Let $\mathcal{R}$ be a  type $\textrm{I}_{n}~ (n\geq 2)$ von Neumann algebra with center $\mathcal{Z}$  and $\tau$ be a faithful normal semi-finite trace on $\mathcal{R}.$   We denote the centers of $S(\mathcal{R})$ and $S(\mathcal{R},\tau)$ by $\mathcal{Z}(S(\mathcal{R}))$ and $\mathcal{Z}(S(\mathcal{R},\tau))$, respectively. By Proposition 1.2 of \cite{AAK1}, we have  $\mathcal{Z}(S(\mathcal{R}))=S(\mathcal{Z})$ and $\mathcal{Z}(S(\mathcal{R},\tau))=S(\mathcal{Z},\tau_{\mathcal{Z}}),$ where $\tau_{\mathcal{Z}}$ is the restriction of the trace $\tau$ on $\mathcal{Z}.$ By Propositions 1.4 and 1.5 of \cite{AAK1}, $S(\mathcal{R})=LS(\mathcal{R})\cong M_{n}(S(\mathcal{Z}))$ and
 $S(\mathcal{R},\tau)\cong M_{n}(S(\mathcal{Z},\tau_{\mathcal{Z}})).$

As a direct application of Theorem \ref{thm 3.6}, we have the following corollary.
\begin{corollary}\label{cor 3.10}
Suppose that  $\mathcal{R}$ is a  type $\textrm{I}_{n}, n\geq 2,$ von Neumann algebra and $\tau$ is a faithful normal semi-finite trace on $\mathcal{R}.$  Then we have

 $(1)$  ~~every 2-local derivation $\Delta: \mathcal{R}\rightarrow LS(\mathcal{R})$ is a derivation;

$(2)$   ~~every 2-local derivation $\Delta: \mathcal{R}\rightarrow S(\mathcal{R},\tau)$ is a derivation.

\end{corollary}

\begin{lemma}\label{lem 3.11}
Let $\Delta:\mathcal{A}\rightarrow \mathcal{M}$ be a 2-local derivation. If there exists a central idempotent $e$ in $\mathcal{A}$ which commutates with  $\mathcal{M},$ then
$\Delta(ea)=e\Delta(a),$ for each $a$ in $\mathcal{A}.$
\end{lemma}
\begin{proof}
 For any $a \in\mathcal{A},$ by assumption, there exists  a derivation $\delta:\mathcal{A}\rightarrow \mathcal{M}$ such that:
$\Delta(ea)=\delta(ea),$  and $\Delta(a)=\delta(a).$  By assumption, $e$ is a central idempotent  in $\mathcal{A}$ which commutes with  $\mathcal{M},$  it follows that $\delta(e)=0.$
Then $$\Delta(ea)=\delta(ea)=\delta(e)a+e\delta(a)=e\delta(a)=e\Delta(a).$$
The proof is complete.
\end{proof}

\begin{theorem}\label{thm 3.12}
Suppose that  $\mathcal{R}$ is a finite  von Neumann algebra of type $\textrm{I}$ without abelian direct summands. Then every 2-local derivation $\Delta: \mathcal{R}\rightarrow S(\mathcal{R})= LS(\mathcal{R})$ is a derivation.
\end{theorem}
\begin{proof}
By assumption, $\mathcal{R}$ is a finite  von Neumann algebra of type $\textrm{I}$ without abelian direct summands. Then there exists a family $\{P_{n}\}_{n\in F},~F\subseteq \mathbb{N}\setminus{1},$ of  orthogonal central projections in $\mathcal{R}$ with $\sum_{n\in F} P_{n}=1,$ such that the algebra $\mathcal{R}$ is $*$-isomorphic with the $C^{*}$-product of von Neumann algebras $P_{n}\mathcal{R}$ of type $\textrm{I}_{n}$, respectively $n\in F.$ Then
 $$P_{n}LS(\mathcal{R})=P_{n}S(\mathcal{R})=S(P_{n}\mathcal{R})\cong M_{n}(P_{n}Z(\mathcal{R})),~n\in F.$$
 By Lemma \ref{lem 3.11}, we have   $\Delta(P_{n}A)=P_{n}\Delta(A),$ for all $A\in\mathcal{R}$ and  each $n\in F.$ This implies that $\Delta$ maps each $P_{n}\mathcal{R}$ into $P_{n}S(\mathcal{R}).$
 For each $n\in F,$ we define $\Delta_{n}:P_{n}\mathcal{R}\rightarrow P_{n}S(\mathcal{R})$ by
 $$\Delta_{n}(P_{n}A)=P_{n}\Delta(A),~A\in \mathcal{R}.$$
 By assumption, it follows that $\Delta_{n}$ is a 2-local derivation from $P_{n}\mathcal{R}$ into $P_{n}S(\mathcal{R})$ for each $n\in F.$ By $(1)$ of Corollary \ref{cor 3.10}, we have that $\Delta_{n}$ is a derivation for each $n\in F.$ Since $\sum_{n\in F}P_{n}=1,$ it follows that $\Delta$ is a linear mapping. For any $A,B\in \mathcal{R},$ it follows $\Delta_{n}$ is a derivation for each $n\in F$  that
 \begin{align*}
 P_{n}\Delta(AB)&=\Delta_{n}(P_{n}AB)=\Delta_{n}(P_{n}A)P_{n}B+P_{n}A\Delta_{n}(P_{n}B)\notag\\
 &=P_{n}\Delta(A)B+P_{n}A\Delta(B)\notag\\
 &=P_{n}(\Delta(A)B+A\Delta(B)).\notag
 \end{align*}
By assumption, $\sum_{n\in F}P_{n}=1,$ we get $$\Delta(AB)=\Delta(A)B+A\Delta(B).$$
Therefore  $\Delta: \mathcal{R}\rightarrow S(\mathcal{R})$ is a derivation. The proof is complete.
\end{proof}
Ayupov, Kudaybergenov and Alauadinov \cite{AKA2} have proved the following result. Now we give a different proof.
\begin{theorem}[\cite{AKA2}, Theorem 3.1]
Let $\mathcal{R}$ be an arbitrary von Neumann algebra without abelian direct summands and $LS(\mathcal{R})$ be the algebra of all locally measurable operators affiliated with $\mathcal{R}.$ Then every 2-local derivation $\Delta: \mathcal{R}\rightarrow LS(\mathcal{R})$ is a derivation.
\end{theorem}
\begin{proof}
Let $\mathcal{R}$ be an arbitrary von Neumann algebra without abelian direct summands. We  known that  $\mathcal{R}$ can be decomposed along a central projection into
the direct sum of von Neumann algebras of finite type $\textrm{I}$, type $\textrm{I}_{\infty},$ type $\textrm{II}$ and type $\textrm{III}.$ By Lemma \ref{lem 3.11}, we may consider these cases separately.

If $\mathcal{R}$ is a von Neumann algebra of finite type $\textrm{I}$, Theorem \ref{thm 3.12} show that every 2-local derivation from $\mathcal{R}$ into  LS$(\mathcal{R})$ is a derivation.

If $\mathcal{R}$ is a von Neumann algebra of types $\textrm{I}_{\infty},$ $\textrm{II}$ or $\textrm{III}.$ Then the halving Lemma (\cite{KR}, Lemma 6.3.3)
for type $\textrm{I}_{\infty}$ algebras and (\cite{KR}, Lemma 6.5.6) for types $\textrm{II}$ or $\textrm{III}$ algebras, imply that the unit of $\mathcal{R}$
can be represented as a sum of mutually equivalent orthogonal projections $e_{1},e_{2},\cdot\cdot\cdot,e_{6}$ in $\mathcal{R}.$ It is well known that $\mathcal{R}$ is isomorphic to
$M_{6}(\mathcal{A}),$ where $\mathcal{A}=e_{1}\mathcal{R}e_{1}.$ Further, the algebra $LS(\mathcal{R})$ is isomorphic to the algebra $M_{6}(LS(\mathcal{A})).$  Theorem \ref{thm 3.9} implies that every 2-local derivation from $\mathcal{R}$ into  LS$(\mathcal{R})$ is a derivation.
The proof is complete.
\end{proof}

\end{document}